\documentclass[]{article}
\usepackage[left=2cm,right=2cm, top=2.4cm,bottom=2.4cm,bindingoffset=0cm]{geometry}
\usepackage{tikz, titlesec, tikz-cd}
\usetikzlibrary{calc,intersections,through,backgrounds}
\usepackage{amsmath, amssymb}
\usepackage{amsthm}

\usepackage{amsthm}
\let\oldproofname=\proofname
\renewcommand{\proofname}{\rm\bf{\oldproofname}}
\usepackage{hyperref}

\usepackage[numbers, sort&compress]{natbib}

\usetikzlibrary{positioning}
\usetikzlibrary{decorations.text}
\usetikzlibrary{decorations.pathmorphing}
\usetikzlibrary{decorations.markings}

\newtheorem{lemma}{Lemma}[section]
\newtheorem{proposition}[lemma]{Proposition}
\newtheorem{theorem}[lemma]{Theorem}
\newtheorem{corollary}[lemma]{Corollary}

\theoremstyle{definition}
\newtheorem{remark}[lemma]{Remark}
\newtheorem{definition}[lemma]{Definition}

\newtheorem{example}[lemma]{Example}

\newcommand{\Z}{\mathbb{Z}}
\newcommand{\C}{\mathbb{C}}

\newcommand{\Hom}{{H}}
\newcommand{\Id}{  I}

\usepackage{bbm}

\begin{document}

\tikzset{->-/.style={decoration={
  markings,
  mark=at position .7 with {\arrow{>}}},postaction={decorate}}}
  
  \tikzset{->>-/.style={decoration={
  markings,
  mark=at position .5 with {\arrow{>}}},postaction={decorate}}}
  
  \tikzset{-<-/.style={decoration={
  markings,
  mark=at position .3 with {\arrow{<}}},postaction={decorate}}}
  
\usetikzlibrary{angles, quotes}

\title{Dimers, networks, and cluster integrable systems}
\author{Anton Izosimov}
\date{}

\maketitle

\abstract{We prove that the class of cluster integrable systems constructed by Goncharov and Kenyon out of the dimer model on a torus coincides with the one defined by Gekhtman, Shapiro, Tabachnikov, and Vainshtein using Postnikov's perfect networks. To that end we express the characteristic polynomial of a perfect network's  boundary measurement matrix  in terms of the dimer partition function of the associated bipartite graph. Our main tool is flat geometry. Namely, we show that if a perfect network is drawn on a flat torus in such a way that the edges of the network are Euclidian geodesics, then the angles between the edges endow the associated bipartite graph with a canonical fractional Kasteleyn orientation. That orientation is then used to relate the partition function to boundary measurements.}


\section{Introduction}

\paragraph{Background.}
This paper deals with two \textit{a priori} different constructions of integrable systems related to cluster algebras. One is due to Goncharov and Kenyon \cite{GK} and is based on the dimer model on a bipartite graph on a torus. The other one is due to Gekhtman, Shapiro, Tabachnikov, and Vainshtein \cite{GSTV} and uses Postnikov's perfect networks. Our main result is that these two constructions produce exactly the same class of integrable systems.
\par
Cluster algebras were introduced by Fomin and Zelevinskiy \cite{FZ}. Gekhtman, Shapiro, and Vainshtein \cite{GSV} defined a family of Poisson structures compatible with the cluster structure. Fock and Goncharov \cite{FG} showed that every $\mathcal X$-type (also known as  $Y$-type) cluster variety has a canonical Poisson structure. Furthermore, in~\cite{GK, GSTV} this Poisson structure was promoted to a completely integrable system. Loosely speaking, both works show that a cluster structure gives rise to an integrable system provided that the corresponding quiver can be drawn on a torus.
The corresponding commuting Hamiltonians are defined using objects which are in a certain sense dual to the quiver. Namely, in~\cite{GK} the dual object is a bipartite graph, while in \,\cite{GSTV} it is a perfect network. Goncharov and Kenyon \cite{GK} conjectured that ``in the cases of interest ... the models are essentially equivalent.'' In the present paper we show that the two models are in fact equivalent {in {all} cases}. This generalizes a series of examples known to fit in both constructions, with the best known example being Schwartz's pentagram map~\cite{Sch}. The cluster structure of the pentagram map was found in \cite{glick2011pentagram}. Its connection with networks is described in  \,\cite{GSTV}. A dimer model interpretation is found in~\cite{FM, AGPR}. 




Fock and Marshakov \cite{FM} also showed that the class of Goncharov-Kenyon systems coincides with the one given by central functions on the loop group of $\mathrm{GL}_n$. As a corollary, we have three equivalent descriptions of the same class of integrable systems: using the dimer model, using networks, and using Poisson-Lie groups. To get from the network description to Poisson-Lie description, one can also use the Poisson property of the boundary measurement map established in \cite{GSV3}. 



\paragraph{The Goncharov-Kenyon system.} To state our main result, we first briefly describe the two constructions. 
We start with the Goncharov-Kenyon system based on the dimer model \cite{GK}. A \textit{toric graph} $\Gamma$ is a graph embedded in a $2$-torus $T^2$ in such a way that its \textit{faces}, i.e. the connected components of its complement $T^2 \setminus \Gamma$, are contractible. A graph is \textit{bipartite} if its vertices are colored black and white in such a way that each edge has one white vertex and one black vertex. In the present paper we only consider bipartite graphs that are \textit{leafless}, i.e. have no univalent vertices.  A \textit{dimer cover} of a bipartite graph (or a \textit{perfect matching}) is a set of edges with the property that every vertex is adjacent to a unique edge of the cover. A \textit{weighted graph} is a graph with numbers assigned to edges (in the present paper the weights are assumed to be complex). Given a dimer cover of a weighted bipartite graph $\Gamma$, its weight is defined as the product of weights of its edges. The sum of weights of all dimer covers of $\Gamma$ is called the \textit{dimer partition function}. It can be computed as the determinant of the so-called \textit{Kasteleyn matrix} \cite{Kas}. \par 
The Goncharov-Kenyon Hamiltonians are basically given by the partition function, modified to take into account the topology of each dimer cover. Namely, assume we are given a toric weighted bipartite graph $\Gamma$. Then, since the edges of $\Gamma$ can be canonically oriented from white to black, its any dimer cover may be viewed as an integral $1$-chain. Furthermore, all such chains have the same boundary, namely the sum of black vertices minus the sum of white vertices. In other words, the difference of two dimer covers is a cycle. Therefore, one can speak about the homology class of a dimer cover. The \textit{Goncharov-Kenyon Hamiltonian} $H_\xi$ corresponding to a class $\xi \in \Hom_1(T^2, \Z)$ is defined as the sum of weights of all dimer covers in the class $\xi$. These Hamiltonians are considered as functions on the space of edge weights up to \textit{gauge transformations}. A gauge transformation is multiplication of weights of all edges adjacent to a given vertex by a given number. Viewing the space of edge weights as the space of $1$-cochains, one identifies its quotient by gauge transformations with the cohomology group $\Hom^1(\Gamma, \C^*)$. Since a gauge transformation multiplies all Goncharov-Kenyon Hamiltonians $H_\xi$ by the same number, the Hamiltonians are well-defined as functions on $\Hom^1(\Gamma, \C^*)$, up to a common factor.  Furthermore, it is shown in \cite{GK} that the space  $\Hom^1(\Gamma, \C^*)$ has a natural Poisson structure such that after a suitable normalization the Hamiltonians $H_\xi$ Poisson-commute. Moreover, if $\Gamma$ satisfies a certain \textit{minimality} condition, then the Hamiltonians $H_\xi$ define a completely integrable system. In what follows, we will not care about completeness and use the term \textit{integrable system} to refer to any collection of Poisson-commuting functions.

\par

Just like the partition function, Goncharov-Kenyon Hamiltonians can be computed from a certain determinant. Namely, as shown in \cite{KOS} (see also Section \ref{sec:GK} below), one can introduce a parameter-dependent version  of the Kasteleyn matrix in such a way that its determinant reads
$$
K(\lambda, \mu)=     \sum\nolimits_{(i,j)} \pm H_{(i,j)}\lambda^i \mu^j
$$
where the summation is over all classes $(i,j) \in \Hom_1(T^2, \Z)$ which contain dimer covers, and $H_{(i,j)}$ is the sum of weights of all dimer covers in the class $(i,j)$. The sign in front of $H_{i,j}$ depends only on the parity of $i$ and $j$. Signs corresponding to three of the four possible parities are positive, while the fourth sign is negative. A particular combination of signs depends on the choice of a \textit{discrete spin structure}. Polynomials $K(\lambda, \mu)$ corresponding to each of the four different spin structures can be obtained from each other by means of a substitution of the form $K(\lambda, \mu) \mapsto K(\pm \lambda, \pm \mu)$. The polynomial $K(\lambda, \mu)$ is called the \textit{characteristic polynomial} of the toric bipartite graph~$\Gamma$. Up to signs, the characteristic polynomial is the generating function of Goncharov-Kenyon Hamiltonians. It is well-defined up to a monomial factor. 
\paragraph{The Gekhtman-Shapiro-Tabachnikov-Vainshtein system.}
We now describe the construction of Gekhtman, Shapiro, Tabachnikov and Vainshtein \cite{GSTV}. It is based on the notion of a perfect network, introduced in the case of a disk by Postnikov \cite{Pos}. A \textit{perfect network} is a weighted directed graph whose vertices are of two type: white and black. White vertices have exactly one incoming edge, while black vertices have exactly one outgoing edge.  As with bipartite graphs, we only consider leafless networks. A \textit{toric perfect network} is a perfect network which is at the same time a toric graph. Given such a network, an \textit{ideal rim} is a simple loop on the torus disjoint from the set of vertices and intersecting the edges in such a way that all intersections have the same sign. Cutting the torus along a rim, one obtains a network on a cylinder (when an edge is cut into two, the weights of the newly formed edges are defined in such a way that their product is equal to the weight of the initial edge). It is no longer perfect in the above sense, because in addition to black and white vertices it has uncolored ones. Each uncolored vertex is either a source or a sink and is located at the boundary of the cylinder. Moreover, all sources are located at one boundary component, while all sinks are at the other one. Such a network is called \textit{a perfect network on a cylinder.} Given such a network, one defines its \textit{boundary measurement matrix} as follows: 
\begin{equation*}
M_{ij}(\lambda) := \sum\nolimits_\gamma (-1)^{c(\gamma)}\lambda^{\mathrm{ind}(\gamma)}\mathrm{wt}(\gamma).
\end{equation*}
Here the sum is taken over all directed paths $\gamma$ from source $i$ to sink $j$. The \textit{weight $\mathrm{wt}(\gamma)$} of $\gamma$ is the product of weights of all edges along $\gamma$. The number $c(\gamma) \in \Z$ is called the \textit{concordance number} of $\gamma$ and is basically the self-intersection index mod $2$ (see Section \ref{sec:pmr}). The number $\mathrm{ind}(\gamma) \in \Z$ is called the index of $\gamma$ and is, roughly speaking, the number of times $\gamma$ goes around the cylinder. More precisely, given a toric network, one considers a simple cycle which has a unique intersection with the rim. Such a cycle is called a \textit{cut}. The image of the cut in the corresponding cylindric network is a path connecting the two boundary components. We orient the cut in such a way that it starts at the component containing the sources and ends at the component containing the sinks. The index of a directed path is then defined as its intersection number with the cut.

In the presence of directed cycles the entries of the boundary measurement matrix may be infinite series. However, they can be always rewritten as rational functions. This is proved in \cite{Pos} in the disk case and \cite{GSV3} in the cylinder case. \par
The Gekhtman-Shapiro-Tabachnikov-Vainshtein Hamiltonians are, basically, coefficients of the characteristic polynomial of the boundary measurement matrix. More precisely, since the characteristic polynomial is a rational function of $\lambda$, the Hamiltonians are defined as the coefficients of its numerator.
%
%
%
  It is shown in \cite{GSTV} that they commute with respect to a Poisson structure on the space of edge weights defined in \cite{GSV2, GSV3}. Furthermore, the Hamiltonians and the Poisson structure are invariant under gauge transformations, so both descend to the quotient space, which is again the cohomology group of the toric network with coefficients in $\C^*$ (also known as the space of \textit{face and trail weights}). Our main result is that the so obtained integrable system on the cohomology coincides with the Goncharov-Kenyon system. The coincidence of Poisson structures was already noted in \cite{GK}. Here we will prove that the Hamiltonians coincide too.\par
\paragraph{From networks to bipartite graphs.}
 \begin{figure}[t]
 \centering
\begin{tikzpicture}[, scale = 1]
\node [draw,circle,color=black, fill=black,inner sep=0pt,minimum size=5pt] (A) at (0,0) {};
\node [draw,circle,color=black, fill=white,inner sep=0pt,minimum size=5pt] (B) at (1,0) {};
\node [draw,circle,color=black, fill=black,inner sep=0pt,minimum size=5pt] (C) at (2,0) {};
\node [draw,circle,color=black, fill=black,inner sep=0pt,minimum size=5pt] (D) at (6,0) {};
\draw (A) -- (B) -- (C);
\draw (A) -- +(-0.5,+0.5);
\draw (A) -- +(-0.5,-0.5);
\draw (C) -- +(+0.5,+0.5);
\draw (C) -- +(+0.5,-0.5);
\node () at (4,0) {$\longleftrightarrow$};
\draw (D) -- +(-0.5,+0.5);
\draw (D) -- +(-0.5,-0.5);
\draw (D) -- +(+0.5,+0.5);
\draw (D) -- +(+0.5,-0.5);

\end{tikzpicture}
\caption{A transformation of bipartite graphs yielding equivalent graphs.}\label{equiv}
\end{figure}
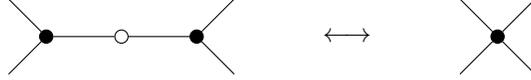
To give a precise statement of our result, we need to explain how to pass from a perfect network to a bipartite graph. We say that two toric bipartite graphs are \textit{equivalent} if they are related by a sequence of $2$-valent vertex removals or additions, see Figure \ref{equiv} (there is an analogous transformation with opposite vertex colors). Clearly, the cohomology groups of equivalent graphs are canonically isomorphic. Moreover, as shown in~\cite{GK}, this isomorphism identifies the corresponding Goncharov-Kenyon integrable systems. \par
Similarly, we say that two perfect networks are {equivalent} if they are related by a sequence of the following trasformations or their inverses:
(a) Insertion of a $2$-valent vertex of any color in the middle of any edge. 
(b)~Contraction of a unicolored edge which is not a loop.
(c) Reversal of an oriented cycle.
As with bipartite graphs, cohomology groups of equivalent networks are isomorphic. Furthermore, we will see below that integrable systems associated with, in some sense, generic networks from a given equivalence class are also isomorphic to each other.

Now, define a map from the set of perfect toric networks to the set of bipartite toric graphs as follows: given a perfect network, insert an opposite color vertex in the middle of every unicolored edge, and forget the orientations. Note that a bipartite graph obtained in this way comes equipped with a dimer covering. The latter is defined by the edges that were oriented from black to white before we forgot the orientations. So we actually have a map from perfect networks to bipartite graphs admitting a dimer covering. It is easy to see that this map is a bijection at the level of the above defined equivalence classes. Also note that if $\mathcal N$  is a perfect network, and $\Gamma$ is the associated bipartite graph, then we have a natural isomorphism $\Hom^1(\mathcal N, \C^*)\simeq\Hom^1(\Gamma, \C^*)$.

\paragraph{The main result.} Given a perfect toric network, one can define its boundary measurement matrix, provided the network admits an ideal rim. We will call perfect networks which have an ideal rim \textit{admissible}. As follows from~\cite[Proposition 4.2]{GSTV}, there is an admissible perfect network in every equivalence class. 


\begin{theorem}\label{thm1}
Consider an admissible perfect toric network $\mathcal N$ with a chosen ideal rim and cut, and let $M(\lambda)$ be its boundary measurement matrix. Let also $\Gamma$ be the toric bipartite graph corresponding to $\mathcal N$, and let $\Psi \colon \Hom^1(\mathcal N, \C^*)\to\Hom^1(\Gamma, \C^*)$ be the natural isomorphism of cohomology groups. 
Then, for an appropriate choice of a spin structure, one has
\begin{equation*}
 \det(\Id - \mu M(\lambda)) = \Psi^*\left( \frac{K(\lambda, \mu)}{K(\lambda, 0)} \right),
\end{equation*}
where $K(\lambda, \mu)$ is the characteristic polynomial of $\Gamma$ written in the homology basis given by the rim and cut of $\mathcal N$ and normalized in such a way that it is a polynomial in $\mu$ not divisible by $\mu$. 
\end{theorem}
It follows that while the Goncharov-Kenyon system is given by the coefficients of $K(\lambda, \mu)$, the Gekhtman-Shapiro-Tabachnikov-Vainshtein system is given by $K(\lambda, \mu) / Q(\lambda)$, where $Q(\lambda)$ is the greatest common divisor of $K(\lambda, \mu)$ and $K(\lambda, 0)$. In particular, if the two latter polynomials are coprime (which is equivalent to saying that $K(\lambda, \mu)$ is not divisible by a non-trivial polynomial of $\lambda$), then the Gekhtman-Shapiro-Tabachnikov-Vainshtein system coincides with the Goncharov-Kenyon system. Since $K(\lambda, \mu)$ has only finitely many irreducible factors, this can be always arranged by adjusting the rim direction. This is always possible, because as follows from \cite[Proposition 4.2]{GSTV}, any homology class on a torus can be taken as the rim direction for a suitable network from a given equivalence class. Therefore, we obtain the following:

\begin{corollary}
Given a  perfect toric network $\mathcal N$, one can choose an equivalent admissible network $\mathcal N'$ so that for a certain rim direction the Gekhtman-Shapiro-Tabachnikov-Vainshtein system associated with $\mathcal N'$ coincides with the Goncharov-Kenyon system associated with the bipartite graph $\Gamma$ corresponding to $\mathcal N$.\par
Conversely, given a toric bipartite graph $\Gamma$, one can find an orientation turning it into an admissible perfect network $\mathcal N$, so that for a certain rim direction the Gekhtman-Shapiro-Tabachnikov-Vainshtein system associated with $\mathcal N$ coincides with the Goncharov-Kenyon system associated with $\Gamma$.
\end{corollary}

Theorem \ref{thm1} assumes a particularly simple form when the network $\mathcal N$ is itself bipartite. In that case, $\Gamma$ is basically the same graph as $\mathcal N$, with the only difference that while $\mathcal N$ has a non-trivial orientation, all edges of~$\Gamma$ are thought of as oriented from white to black. As a result, the map $\Psi \colon \Hom^1(\mathcal N, \C^*)\to\Hom^1(\Gamma, \C^*)$ boils down to the following map between edge weight spaces: replace all weights of black-to-white edges by their reciprocals while keeping all weights of white-to-black edges intact, see Example \ref{exmain} below.


\begin{example}\label{exmain}
 \begin{figure}[t]
 \centering
\begin{tikzpicture}[, scale = 0.9]
\node [draw,circle,color=black, fill=black,inner sep=0pt,minimum size=5pt] (B) at (0,1) {};
\node [draw,circle,color=black, fill=white,inner sep=0pt,minimum size=5pt] (C) at (0,2) {};
\node [draw,circle,color=black, fill=white,inner sep=0pt,minimum size=5pt] (B1) at (1,1) {};
\node [draw,circle,color=black, fill=black,inner sep=0pt,minimum size=5pt] (C1) at (1,2) {};

\draw [->-] (B) -- (B1) node[midway, above] {$e_4$};
\draw [->-] (C) -- (C1)  node[midway, above] {$e_3$};
\draw [->-] (C) -- (B)  node[midway, left] {$e_2$};
\draw [->-] (C) -- +(0,1)  node[midway, left] {$e_1$};
\draw [-<-] (B) -- +(0,-1);
\draw [-<-] (B) -- +(-1,0)  ;
\draw [-<-] (C) -- +(-1,0) ;
\draw [->-] (B1) -- +(1,0) node[midway, above] {$e_6$};  ;
\draw [->-] (C1) -- +(1,0)  node[midway, above] {$e_5$} ;
\draw [dashed] (-1,3) -- (2,3) -- (2,0);
\draw [->>-, dashed] (-1,0) -- (2,0)  node[midway, below] {cut};;
\draw [->>-, dashed] (-1,0) -- (-1,3)  node[midway, left] {rim};;
\end{tikzpicture}
\caption{A perfect toric network.}\label{ex1}
\end{figure}
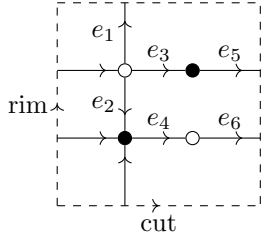
 \begin{figure}[b]
 \centering
 \begin{tikzpicture}[, scale = 0.9]
\node [draw,circle,color=black, fill=black,inner sep=0pt,minimum size=5pt] (B) at (0,1) {};
\node [draw,circle,color=black, fill=white,inner sep=0pt,minimum size=5pt] (C) at (0,2) {};
\node [draw,circle,color=black, fill=white,inner sep=0pt,minimum size=5pt] (B1) at (1,1) {};
\node [draw,circle,color=black, fill=black,inner sep=0pt,minimum size=5pt] (C1) at (1,2) {};

\draw [line width = 2.5pt]  (B) -- (B1) node[midway, above] {$e_4$};
\draw(C) -- (C1)  node[midway, above] {$e_3$};
\draw  (C) -- (B)  node[midway, left] {$e_2$};
\draw (C) -- +(0,1)  node[midway, left] {$e_1$};
\draw  (B) -- +(0,-1);
\draw (B) -- +(-1,0) ;
\draw  [line width = 2.5pt]  (C) -- +(-1,0) ;
\draw  (B1) -- +(1,0)   node[midway, above] {$e_6$};
\draw  [line width = 2.5pt]   (C1) -- +(1,0)  node[midway, above] {$e_5$};
\draw [dashed] (-1,3) -- (2,3) -- (2,0);
\draw [->>-, dashed] (-1,0) -- (2,0)  node[midway, below] {$\gamma_2$};;
\draw [->>-, dashed] (-1,0) -- (-1,3)  node[midway, left] {$\gamma_1$};;
\node  () at (0.5,-0.7) {weight = $x_4x_5$};
\node  () at (0.5,-1.2) {class = $0$};
\end{tikzpicture} \quad
\begin{tikzpicture}[, scale = 0.9]
\node  () at (0.5,-0.7) {weight = $x_3x_4$};
\node  () at (0.5,-1.2) {class = $[\gamma_2]$};
\node [draw,circle,color=black, fill=black,inner sep=0pt,minimum size=5pt] (B) at (0,1) {};
\node [draw,circle,color=black, fill=white,inner sep=0pt,minimum size=5pt] (C) at (0,2) {};
\node [draw,circle,color=black, fill=white,inner sep=0pt,minimum size=5pt] (B1) at (1,1) {};
\node [draw,circle,color=black, fill=black,inner sep=0pt,minimum size=5pt] (C1) at (1,2) {};

\draw [line width = 2.5pt]  (B) -- (B1) node[midway, above] {$e_4$};
\draw [line width = 2.5pt] (C) -- (C1)  node[midway, above] {$e_3$};
\draw  (C) -- (B)  node[midway, left] {$e_2$};
\draw (C) -- +(0,1)  node[midway, left] {$e_1$};
\draw  (B) -- +(0,-1);
\draw (B) -- +(-1,0) ;
\draw (C) -- +(-1,0)  ;
\draw  (B1) -- +(1,0)  node[midway, above] {$e_6$} ;
\draw  (C1) -- +(1,0) node[midway, above] {$e_5$};
\draw [dashed] (-1,3) -- (2,3) -- (2,0);
\draw [->>-, dashed] (-1,0) -- (2,0)  node[midway, below] {$\gamma_2$};;
\draw [->>-, dashed] (-1,0) -- (-1,3)  node[midway, left] {$\gamma_1$};;
\end{tikzpicture} \quad
\begin{tikzpicture}[, scale = 0.9]
\node  () at (0.5,-0.7) {weight = $x_5x_6$};
\node  () at (0.5,-1.2) {class = $[\gamma_2]$};
\node [draw,circle,color=black, fill=black,inner sep=0pt,minimum size=5pt] (B) at (0,1) {};
\node [draw,circle,color=black, fill=white,inner sep=0pt,minimum size=5pt] (C) at (0,2) {};
\node [draw,circle,color=black, fill=white,inner sep=0pt,minimum size=5pt] (B1) at (1,1) {};
\node [draw,circle,color=black, fill=black,inner sep=0pt,minimum size=5pt] (C1) at (1,2) {};

\draw (B) -- (B1) node[midway, above] {$e_4$};
\draw (C) -- (C1)  node[midway, above] {$e_3$};
\draw  (C) -- (B)  node[midway, left] {$e_2$};
\draw (C) -- +(0,1)  node[midway, left] {$e_1$};
\draw  (B) -- +(0,-1);
\draw  [line width = 2.5pt]   (B) -- +(-1,0)  ;
\draw[line width = 2.5pt]  (C) -- +(-1,0)  ;
\draw  [line width = 2.5pt]   (B1) -- +(1,0)  node[midway, above] {$e_6$};
\draw [line width = 2.5pt]   (C1) -- +(1,0) node[midway, above] {$e_5$};
\draw [dashed] (-1,3) -- (2,3) -- (2,0);
\draw [->>-, dashed] (-1,0) -- (2,0)  node[midway, below] {$\gamma_2$};;
\draw [->>-, dashed] (-1,0) -- (-1,3)  node[midway, left] {$\gamma_1$};;
\end{tikzpicture}
\quad
\begin{tikzpicture}[, scale = 0.9]
\node  () at (0.5,-0.7) {weight = $x_3x_6$};
\node  () at (0.5,-1.2) {class = $2[\gamma_2]$};
\node [draw,circle,color=black, fill=black,inner sep=0pt,minimum size=5pt] (B) at (0,1) {};
\node [draw,circle,color=black, fill=white,inner sep=0pt,minimum size=5pt] (C) at (0,2) {};
\node [draw,circle,color=black, fill=white,inner sep=0pt,minimum size=5pt] (B1) at (1,1) {};
\node [draw,circle,color=black, fill=black,inner sep=0pt,minimum size=5pt] (C1) at (1,2) {};

\draw (B) -- (B1) node[midway, above] {$e_4$};
\draw [line width = 2.5pt] (C) -- (C1)  node[midway, above] {$e_3$};
\draw  (C) -- (B)  node[midway, left] {$e_2$};
\draw (C) -- +(0,1)  node[midway, left] {$e_1$};
\draw  (B) -- +(0,-1);
\draw  [line width = 2.5pt]   (B) -- +(-1,0)  ;
\draw (C) -- +(-1,0)  ;
\draw  [line width = 2.5pt]   (B1) -- +(1,0) node[midway, above] {$e_6$} ;
\draw  (C1) -- +(1,0) node[midway, above] {$e_5$};
\draw [dashed] (-1,3) -- (2,3) -- (2,0);
\draw [->>-, dashed] (-1,0) -- (2,0)  node[midway, below] {$\gamma_2$};;
\draw [->>-, dashed] (-1,0) -- (-1,3)  node[midway, left] {$\gamma_1$};;
\end{tikzpicture}
 
\caption{Dimer covers of the bipartite graph shown in Figure \ref{ex1}.}\label{ex11}
\end{figure}
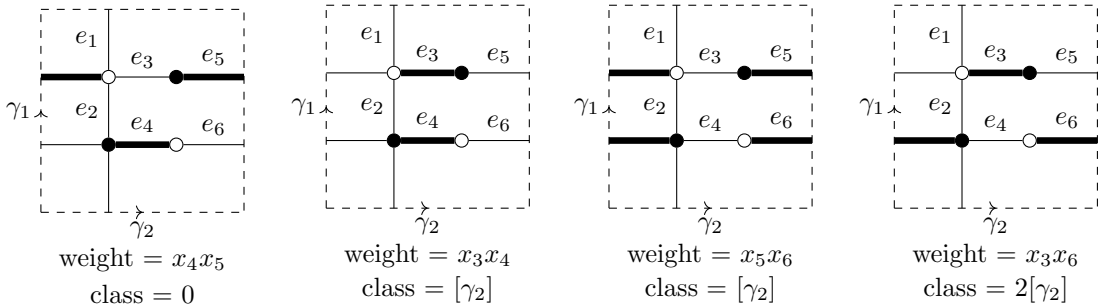

Figure \ref{ex1} shows a perfect toric network $\mathcal N$ (the opposite sides of the square are identified). Choose the rim and cut as shown in the figure and label the sources and sinks from top down. Denote the weight of each edge $e_i$ by $x_i$. Then the boundary measurement matrix is
$$
M(\lambda) = \left(\begin{array}{cc}x_3x_5 & x_2x_4x_6 + x_1x_4x_6 \lambda \\0 & x_4x_6\end{array}\right).
$$
Consider now the associated bipartite graph $\Gamma$. It is the same graph as $\mathcal N$ but without orientations. It has four dimer covers shown in Figure \ref{ex11}. Under each cover we provide its weight and homology class (the homology classes are defined up to shift). One can choose the spin structure so that the characteristic polynomial reads
$$
K(\lambda, \mu) = x_4x_5 - (x_3x_4 + x_5x_6)\mu + x_3x_6 \mu^2.
$$
Then
$$
\frac{K(\lambda, \mu)}{K(\lambda, 0)} = 1 - \left(\frac{x_3}{x_5} + \frac{x_6}{x_4}\right)\mu + \frac{x_3x_6}{x_4x_5}   \mu^2 = \left( 1 -  \frac{x_3}{x_5} \mu\right)\left( 1 -  \frac{x_6}{x_4} \mu\right).
$$
The natural isomorphism $\Psi \colon \Hom^1(\mathcal N, \C^*)\to\Hom^1(\Gamma, \C^*)$ is induced by the following map between edge weight spaces:
$$
(x_1, x_2, x_3, x_4, x_5, x_6) \mapsto (x_1, x_2, x_3, x_4^{-1}, x_5^{-1}, x_6),
$$
so
$$
\Psi^*\left( \frac{K(\lambda, \mu)}{K(\lambda, 0)} \right) =  (1 - x_3x_5 \mu )( 1 - x_4x_6\mu) = \det(\Id - \mu M(\lambda)).
$$

Note that in this example $K(\lambda, \mu)$ is not divisible by a non-trivial polynomial of $\lambda$, so the two integrable systems coincide. The corresponding Hamiltonians are given by the sum and product of the functions on the cohomology given by pairing with the cycles $e_3 + e_5$ and $e_4 + e_6$.
 This is, however, no longer the case if we switch the rim with the cut. In this new basis, the dimer characteristic polynomial is a polynomial of $\lambda$, so the characteristic polynomial of the boundary measurement matrix is trivial. And indeed, there is a single source and a single sink, with no directed paths going from the source to the sink. So the boundary measurement matrix vanishes and $ \det(\Id - \mu M(\lambda)) = 1$. In this case, the Goncharov-Kenyon system is still the same (as it is basis independent), while the Gekhtman-Shapiro-Tabachnikov-Vainshtein system is trivial. 

\end{example}

\paragraph{Outline of the proof.} The rest of the paper is devoted to the proof of Theorem \ref{thm1}. Our proof strategy is as follows. First, we know that boundary measurements do not change when we insert $2$-valent vertices. So, it is sufficient to prove the theorem for bipartite networks. Moreover, one can assume that after cutting the torus along the rim all vertices connected to sources are white, while all vertices connected to sinks are black. This again can be arranged by inserting additional vertices. Finally, inserting additional vertices if needed, one can assume that the network can be drawn on a \textit{flat} torus in such a way that all edges are straight lines, and moreover all edges intersecting the rim are parallel to each other. This flat structure is the main tool we use to establish the equivalence of two constructions. 
Namely, we show that if a bipartite perfect network is drawn on a flat torus, then the angles between its edges endow it with a canonical \textit{fractional Kasteleyn orientation}. This orientation, on one hand, determines the concordance numbers needed to calculate the boundary measurement matrix, and on the other hand allows one to compute the dimer characteristic polynomial. As a result, we obtain the stated relation between the two objects.


\paragraph{Acknowledgements.} The author is grateful to Michael Gekhtman and Pavlo Pylyavskyy for fruitful conversations and useful remarks. This work was supported by NSF grant DMS-2008021.
\medskip

\section{Determinantal expression for the characteristic polynomial of the dimer model}\label{sec:GK}
In this section we recall the determinantal description of the characteristic polynomial $K(\lambda, \mu)$ of the dimer model. Our approach is similar to that of \cite[Section 2.1]{KO} and \cite[Section 5.6]{FM}. 

\paragraph{Kasteleyn orientations and markings.} Consider a toric bipartite graph $\Gamma$. For notational convenience, assume that every face of $\Gamma$ is an embedded polygon (if not, then the closure of every face $f$ of $\Gamma$ can still be thought as the image of a polygon $ f'$ under a cellular map, and in the following definitions instead of counting vertices and edges of $f$ one needs to count vertices and edges of $ f'$).
\begin{definition}
An orientation of a toric bipartite graph is called a \textit{Kasteleyn orientation} if it satisfies one of the following equivalent conditions:
\begin{enumerate} \item For every face, the number of clockwise oriented boundary edges is odd (since every face of a bipartite graph has even number of vertices, this is equivalent to saying that the number of counter-clockwise oriented boundary edges is odd). 
\item
Every $4k$-gonal face has odd number of black-to-white oriented edges (and hence odd number of white-to-black oriented edges), while every $(4k + 2)$-gonal face has even number of black-to-white oriented edges  (and hence even number  of white-to-black oriented edges).
\end{enumerate}
\end{definition}
\begin{definition}
A \textit{Kasteleyn marking} on a toric bipartite graph is an assignment of $\pm 1$ to every edge such that for every face $f$ the product of signs over its edges is equal to $(-1)^{l(f)/2+1}$ where $l(f)$ is the number of vertices of $f$. 
\end{definition}
Kasteleyn markings are in one-to-one correspondence with Kasteleyn orientations. To get a Kasteleyn marking from a Kasteleyn orientation, one assigns $1$ to edges oriented from white to black, and $-1$ to edges oriented from black to white.
\paragraph{The Kasteleyn operator and characteristic polynomial.} Consider a weighted toric bipartite graph~$\Gamma$. Choose some Kasteleyn orientation (equivalently,  Kasteleyn marking) of $\Gamma$. 
Consider the lift $\tilde \Gamma$ of $\Gamma$ to the universal covering of the torus, and let $B$ (respectively, $W$) be the set of black (respectively, white) vertices of $\tilde \Gamma$. The \textit{Kasteleyn operator} $\mathcal K \colon \C^B \to \C^W$ between the corresponding function spaces is defined as follows. For a white vertex $w \in W$, let $e_1, \dots, e_m$ be the edges of $\tilde \Gamma$ incident to $w$, and let $b_1, \dots, b_m$ be their other endpoints. For an edge $e$ of $\tilde \Gamma$, denote by $\mathrm{wt}(e)$ the weight of $e$, and by $k(e)$ the Kasteleyn marking of $e$ (the weights and markings for the covering $\tilde \Gamma$ are defined by pulling back the corresponding objects from $\Gamma$). Then, for a function $g \colon B \to \C$, one sets
$$
(\mathcal K(g))(w) := \sum\nolimits_{j=1}^m k(e_j) \mathrm{wt}(e_j) g(b_j).
$$
Let $\gamma_1, \gamma_2$ be simple oriented cycles on the torus intersecting at one point and disjoint from the vertices of~$\Gamma$. Their homology classes form a basis in $\Hom_1(T^2, \Z)$ and also give rise to a basis $T_1, T_2$ of the group of deck transformations for the universal covering of the torus. Let
$
\C^B_{\lambda, \mu}  := \{g \in \C^B \mid T_1^*g = \lambda g; T_2^*g = \mu g\}
$
be the space of quasi-periodic functions on black vertices with monodromies $\lambda, \mu$. Analogously, one defines the space $\C^W_{\lambda, \mu} $ of quasi-periodic functions on white vertices. The Kasteleyn operator restricts to a linear map $ \C^B_{\lambda, \mu} \to \C^W_{\lambda, \mu}$, which we denote by $\mathcal K(\lambda, \mu)$. Note that the dimension of $  \C^B_{\lambda, \mu}$ is equal to the number of black vertices of $\Gamma$, while the dimension of $  \C^W_{\lambda, \mu}$ is equal to the number of white vertices, so the dimensions are the same as long as the graph $\Gamma$ admits at least one dimer covering (which we from now on assume to be the case). In particular, the determinant of the operator  $\mathcal K(\lambda, \mu)$ is well-defined.
\begin{definition}
The characteristic polynomial of $\Gamma$ (relative to the homology basis given by $\gamma_1$, $\gamma_2$) is
$$
K(\lambda, \mu) := \det \mathcal K(\lambda, \mu).
$$
\end{definition}
As we show below, for a suitable choice of bases in the spaces $\C^B_{\lambda, \mu}$ and  $\C^W_{\lambda, \mu}$ this function is a Laurent polynomial in terms of $\lambda, \mu$. Since the zero locus of $K(\lambda, \mu)$ in the torus $(\C^*)^2$ is basis-independent, it follows that $K(\lambda, \mu)$ is well-defined up to a monomial factor. Furthermore, up to a monomial factor it coincides with the characteristic polynomial defined in the introduction, i.e. it is a sign-twisted generating function of sums of weights of dimer covers in each homology class. The latter statement is essentially the content of Kasteleyn's theorem \cite{Kas}.\par

The bases in the spaces  $\C^B_{\lambda, \mu}$ and  $\C^W_{\lambda, \mu}$ are constructed as follows. Number the black and white vertices of $\Gamma$ by $1, \dots, n$ (recall that the numbers of black and white vertices are assumed to be the same). Choose a fundamental domain $\Omega$ of the universal covering of the torus bounded by preimages of the cycles $\gamma_1$, $\gamma_2$. Let $g_i^b \colon B \to \C$ be the function that takes the value $1$ at the vertex $b_i \in \Omega$ corresponding to the black vertex $i$ of $\Gamma$, and value $0$ at other black vertices $b \in \Omega$. Then $g_1^b, \dots, g_n^b$ is a basis in $\C^B_{\lambda, \mu}$. Furthermore, the same choice of a fundamental domain determines a basis $g_1^w, \dots, g_n^w$ in $\C^W_{\lambda, \mu}$. The matrix of the operator  $\mathcal K(\lambda, \mu)$ written in those bases is given by
\begin{equation}\label{makm}
\mathcal K_{ij}(\lambda, \mu) = \sum\nolimits_e k(e) \mathrm{wt}(e) \lambda^{\langle e , \gamma_2 \rangle} \mu^{\langle \gamma_1, e\rangle}
\end{equation}
where the sum is taken over all edges going from $i$'th white to $j$'th black vertex of $\Gamma$, and $\langle a, b \rangle$ stands for the intersection number of the curves $a$, $b$. We assume that all intersections are transversal, and that the orientation is chosen in such a way that $\langle \gamma_1, \gamma_2 \rangle = 1$. As before, the edges of $\Gamma$ are oriented from white to black. The matrix $\mathcal K_{ij}(\lambda, \mu) $ is a version of the \textit{magnetically altered Kasteleyn matrix} of \cite{KOS}, with slightly different sign conventions. It is manifestly a Laurent polynomial in $\lambda, \mu$, and hence so is its determinant $K(\lambda, \mu)$.

\paragraph{Dependence of the characteristic polynomial on the  Kasteleyn orientation.} Notice that the ratio of two Kasteleyn markings is a $\Z_2$-valued $1$-cocycle on the torus, so by choosing a reference marking one can identify the space of Kasteleyn markings with the space of $\Z_2$-valued $1$-cocycles for the cellular decomposition of $T^2$ given by the graph~$\Gamma$. Furthermore, it is easy to see that changing a Kasteleyn marking by a coboundary does not affect the characteristic polynomial $K(\lambda, \mu)$ (up to a factor independent of $\lambda$ and $\mu$). So, the characteristic polynomial only depends on the cohomology class of a Kasteleyn marking, also known as a \textit{discrete spin structure}~\cite{CR}. By choosing a reference spin structure, one can identify the space of spin structures with the cohomology group  $\Hom^1(T^2, \Z_2)$. Upon a change of the spin structure, the characteristic polynomial transforms as $K(\lambda, \mu) \mapsto K(\pm \lambda, \pm \mu)$. There are four different spin structures, one for each of the four possible combinations of signs. 
\medskip
\section{A fractional Kasteleyn marking from turning numbers}
\paragraph{Fractional Kasteleyn markings.} In order to relate the  characteristic polynomial of the dimer model to boundary measurements, we extend the definition of a Kasteleyn orientation/marking to allow for fractional markings. The following is equivalent to the notion of a \textit{Kasteleyn line bundle with connection} defined in \cite[Section 1.4]{GK}.
\begin{definition}
A \textit{fractional Kasteleyn marking} on a toric bipartite graph is an assignment of a non-zero complex number to every edge such that:

\begin{enumerate} \item For every face $f$ the alternating product of markings around $f$ is equal to $(-1)^{l(f)/2+1}$ where $l(f)$ is the number of vertices of $f$. \item The alternating product of markings over any cycle is $\pm 1$ (this follows from the first condition for contractible cycles).
\end{enumerate}

\end{definition}
It is clear that one can define the Kasteleyn operator and the characteristic polynomial using a fractional Kasteleyn marking instead of an integral one. The characteristic polynomial is still well-defined, up to a transformation of the form $K(\lambda, \mu) \mapsto K(\pm \lambda, \pm \mu)$.

\paragraph{A canonical fractional Kasteleyn marking from turning numbers.} We now give a construction of a special fractional Kasteleyn marking which is well suited for our purposes. From now on, we assume that the graph $\Gamma$ is drawn on a flat torus, with straight edges, and is obtained from a bipartite network.  The latter can be reformulated by saying that $\Gamma$ is endowed with a \textit{perfect orientation}, i.e. an orientation such that any white vertex has exactly one incoming edge, and any black vertex has exactly one outgoing edge. Such a structure is equivalent to a dimer cover. Indeed, given a dimer cover one obtains a perfect orientation by orienting all edges of the cover from black to white, and all other edges from white to black. And conversely, black-to-white edges of a perfect orientation form a dimer cover. \par

Given a bipartite graph on a flat torus and its perfect orientation, one obtains a fractional Kasteleyn marking as follows. Consider an edge $e$ oriented from white to black. Such an edge has a unique predecessor $e_-$ and a unique successor $e_+$ (see Figure \ref{tn}). Let $\alpha_- \in (-\pi, \pi)$ be the signed angle between the vectors $e_-$ and $e$ (note that the angle between two successive edges cannot be equal to $\pi$). Likewise, let $\alpha_+ \in (-\pi, \pi)$ be the signed angle between $e$ and $e_+$. 

 \begin{figure}[t]
 \centering
\begin{tikzpicture}[, scale = 1]
\node (A) at (0,0.5) {};
\node [draw,circle,color=black, fill=white,inner sep=0pt,minimum size=5pt] (B) at (-0.5,2) {};
\node [draw,circle,color=black, fill=black,inner sep=0pt,minimum size=5pt] (C) at (2,2) {};

\node () at (6.5,2) { $\displaystyle\mathrm{turn}(e) := \exp\left( \frac{\mathrm i}{2}(\alpha_- + \alpha_+)\right)$.
};
            \draw [->-] (A) -- (B) node[midway, right] {$e_-$};

            \draw [dashed] (B) -- +(-0.3, 0.9) coordinate (Bp);
                        \pic [draw, 
      angle radius=3mm, angle eccentricity=1.7, 
      "$\alpha_-$", <-] {angle =C--B--Bp};
            \draw [->-] (B) -- (C) node[midway, above] {$e$};
                        \draw [->-] (B) -- +(-1,-1);
                                   \draw [-<-] (C) -- +(1,-1);
                               \draw [->] (C) -- (2.7, 3.2) coordinate (Ca)  node[midway, left] {$e_+$};  
                            \draw [dashed] (C) -- +(1, +0) coordinate (Cpp);     
                                      \pic [draw, 
      angle radius=5mm, angle eccentricity=1.6, 
      "$\alpha_+$", ->] {angle =Cpp--C--Ca};   

\end{tikzpicture}
\caption{To the definition of the turning number.}\label{tn}
\end{figure}
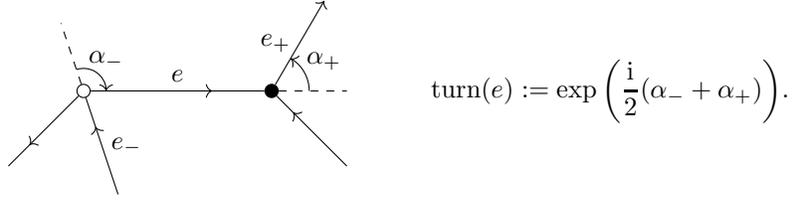

\begin{definition} The \textit{turning number} of a white-to-black edge $e$ is
$$
\mathrm{turn}(e) :=  \exp\left( \frac{\mathrm i}{2}(\alpha_- + \alpha_+)\right).
$$
\end{definition}

\begin{proposition}\label{prop:Kast}
The assignment of the turning number to each white-to-black edge and $-1$ to each black-to-white edge is a fractional Kasteleyn marking.
\end{proposition}
\begin{remark}
Recall that in terms of Kasteleyn orientations, $-1$ means from black to white. That means we keep the orientation of black-to-white edges unchanged. As for white-to-black edges, they get fractional orientations.
\end{remark}
The proof of Proposition \ref{prop:Kast} is given at the end of this section.
\begin{example}\label{excycle}
 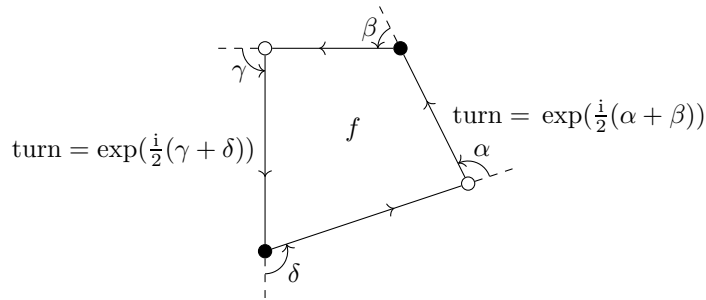
\begin{figure}[b]
 \centering
\begin{tikzpicture}[, scale = 0.9, decoration = {snake,   
                    pre length=3pt,post length=7pt,
                    },]

\node [draw,circle,color=black, fill=white,inner sep=0pt,minimum size=5pt] (A) at (3,0) {};
\node [draw,circle,color=black, fill=black,inner sep=0pt,minimum size=5pt] (B) at (2,2) {};
\node [draw,circle,color=black, fill=white,inner sep=0pt,minimum size=5pt] (C) at (0,2) {};
\node [draw,circle,color=black, fill=black,inner sep=0pt,minimum size=5pt] (D) at (0,-1) {};

     \node  () at (1.3,0.8) {$f$};   

\draw [->-] (A) -- (B) node[midway, right] {$\,\,\mathrm{turn} = \,\exp(\frac{\mathrm i}{2}(\alpha + \beta))$};   ;
\draw [->-] (B) -- (C);
\draw [->-] (C) -- (D)node[midway, left] {$\mathrm{turn} = \exp(\frac{\mathrm i}{2}(\gamma + \delta))$};  ;
\draw [->-] (D) -- (A);
\draw [dashed] (A) -- +(0.75,0.25) coordinate (Ap);
                   \pic [draw, 
      angle radius=3mm, angle eccentricity=1.6, 
      "$\alpha$", ->] {angle =Ap--A--B};
\draw [dashed] (B) -- +(-0.35,0.7) coordinate (Bp);
\draw [dashed] (C) -- +(-0.75,0) coordinate (Cp);
\draw [dashed] (D) -- +(0,-0.75)coordinate (Dp);
                   \pic [draw, 
      angle radius=3mm, angle eccentricity=1.6, 
      "$\beta$", ->] {angle =Bp--B--C};
                         \pic [draw, 
      angle radius=3mm, angle eccentricity=1.6, 
      "$\gamma$", ->] {angle =Cp--C--D};
                         \pic [draw, 
      angle radius=3mm, angle eccentricity=1.6, 
      "$\delta$", ->] {angle =Dp--D--A};
%

\end{tikzpicture}
\caption{Turning numbers in a directed cycle.}\label{Kas3}
\end{figure}
Assume that all edges of a face $f$ are oriented in the same direction. Then, since every second edge is white-to-black, and every other edge is black-to-white, the alternating product of markings is equal to the product of turning numbers of white-to-black edges times $(-1)^{l(f)/2}$. Furthermore, for any white-to-black edge $e$ of $f$ we have $$\mathrm{turn}(e) = \exp\left(\frac{\mathrm i}{2}(\mathrm{ext}_+(e) + \mathrm{ext}_-(e))\right),$$ where $\mathrm{ext}_\pm(e)$ are exterior angles of $f$ adjacent to~$e$, 
see Figure \ref{Kas3}. And since the sum of all exterior angles is $2\pi$, the product of turning numbers is $\exp(\pi \mathrm i) = -1$. Therefore, the alternating product of all markings is indeed~$
(-1)^{l(f)/2 + 1}.
$
\end{example}

\begin{example}\label{exkm}
 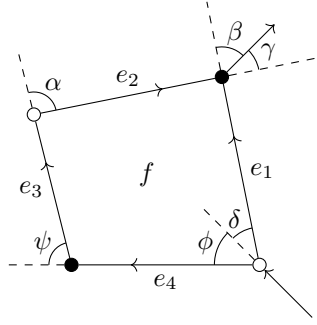
\begin{figure}[t]
 \centering
\begin{tikzpicture}[, scale = 1]
\node [draw,circle,color=black, fill=black,inner sep=0pt,minimum size=5pt] (A) at (0,0) {};
\node [draw,circle,color=black, fill=white,inner sep=0pt,minimum size=5pt] (B) at (-0.5,2) {};
\node [draw,circle,color=black, fill=black,inner sep=0pt,minimum size=5pt] (C) at (2,2.5) {};
\node [draw,circle,color=black, fill=white,inner sep=0pt,minimum size=5pt] (D) at (2.5,-0) {};
            \draw [->-] (A) -- (B) node[midway, left] {$e_3$};

            \draw [dashed] (B) -- (-0.7, 2.8) coordinate (Bp);
                        \pic [draw, 
      angle radius=3mm, angle eccentricity=1.6, 
      "$\alpha$"] {angle =C--B--Bp};
            \draw [->-] (B) -- (C) node[midway, above] {$e_2$};
            \draw [->-] (D) -- (A) node[midway, below] {$e_4$};   
                 \draw [->-] (D) -- (C) node[midway, right] {$e_1$};   
                 \draw [<-]  (D) -- +(0.7, -0.7) ;
                               \draw [->] (C) -- (2.7, 3.2) coordinate (Ca);  
                     \draw [dashed] (C) -- +(-0.2, +1) coordinate (Cp);     
                            \draw [dashed] (C) -- +(1.25, +0.25) coordinate (Cpp);     
                                                 \pic [draw, 
      angle radius=4mm, angle eccentricity=1.5, 
      "$\beta$"] {angle =Ca--C--Cp};   
                                      \pic [draw, 
      angle radius=5mm, angle eccentricity=1.4, 
      "$\gamma$"] {angle =Cpp--C--Ca};   
                                                   \draw [dashed] (D) -- +(-0.8, 0.8)  coordinate (Dp);   
                                                                                         \pic [draw, 
      angle radius=5mm, angle eccentricity=1.4, 
      "$\delta$"] {angle =C--D--Dp};   
                 \pic [draw, 
      angle radius=6mm, angle eccentricity=1.3, 
      "$\phi$"] {angle =Dp--D--A};   
                                                           \draw [dashed] (A) -- +(-0.9,0)  coordinate (Ap); 
                                                                            \node  () at (1,1.2) {$f$};   
                                                                                           \pic [draw, 
      angle radius=3mm, angle eccentricity=1.6, 
      "$\psi$"] {angle =B--A--Ap};   

\end{tikzpicture}
\caption{To Example \ref{exkm}.}\label{Kas}
\end{figure}
Consider a face shown in Figure \ref{Kas}. Here $\alpha, \beta, \gamma, \delta, \phi, \psi \in [0, \pi)$ are unsigned angles. The turning numbers of white-to-black edges are
$$
 \mathrm{turn}(e_1) =  \exp\left(-\frac{\mathrm i}{2}({\beta + \delta})\right), \quad  \mathrm{turn}(e_2) =  \exp\left(\frac{\mathrm i}{2}({\gamma - \alpha})\right), \quad   \mathrm{turn}(e_4) =  \exp\left(\frac{\mathrm i}{2}({\phi - \psi})\right),
$$
so the alternating product of markings  is
\begin{gather*}
-\frac{ \mathrm{turn}(e_1) \mathrm{turn}(e_3)}{ \mathrm{turn}(e_4)} = -\exp\left({\frac{\mathrm i}{2}(-(\gamma  + \beta) - (\delta + \phi) +\alpha + \psi)}\right). 
\end{gather*}
The sum of two exterior angles of a quadrilateral is equal to the sum of two non-adjacent interior angles, so the latter expression is equal to $-1$, in agreement with Proposition \ref{prop:Kast}.

\end{example}

\begin{example}\label{ex:sg}
 \begin{figure}[b]
 \centering
\begin{tikzpicture}[, scale = 1]
\node [draw,circle,color=black, fill=white,inner sep=0pt,minimum size=5pt] (A) at (0,0) {};
\node [draw,circle,color=black, fill=black,inner sep=0pt,minimum size=5pt] (B) at (0,1) {};
\node [draw,circle,color=black, fill=white,inner sep=0pt,minimum size=5pt] (C) at (0,2) {};
\node [draw,circle,color=black, fill=black,inner sep=0pt,minimum size=5pt] (A1) at (1,0) {};
\node [draw,circle,color=black, fill=white,inner sep=0pt,minimum size=5pt] (B1) at (1,1) {};
\node [draw,circle,color=black, fill=black,inner sep=0pt,minimum size=5pt] (C1) at (1,2) {};
\node [draw,circle,color=black, fill=white,inner sep=0pt,minimum size=5pt] (A2) at (2,0) {};
\node [draw,circle,color=black, fill=black,inner sep=0pt,minimum size=5pt] (B2) at (2,1) {};
\node [draw,circle,color=black, fill=white,inner sep=0pt,minimum size=5pt] (C2) at (2,2) {};
\node [draw,circle,color=black, fill=black,inner sep=0pt,minimum size=5pt] (A3) at (3,0) {};
\node [draw,circle,color=black, fill=white,inner sep=0pt,minimum size=5pt] (B3) at (3,1) {};
\node [draw,circle,color=black, fill=black,inner sep=0pt,minimum size=5pt] (C3) at (3,2) {};

\draw [->-] (A) -- (A1);
\draw [->-] (A1) -- (A2);
\draw [->-] (A2) -- (A3);
\draw [->-] (B) -- (B1);
\draw [->-] (B1) -- (B2);
\draw [->-] (B2) -- (B3);
\draw [->-] (C) -- (C1);
\draw [->-] (C1) -- (C2);
\draw [->-] (C2) -- (C3);
\draw [->-] (A) -- (B);
\draw [->-] (C) -- (B);
\draw [->-] (A2) -- (B2);
\draw [->-] (C2) -- (B2);
\draw [->-] (B1) -- (A1);
\draw [->-] (B1) -- (C1);
\draw [->-] (B3) -- (A3);
\draw [->-] (B3) -- (C3);
\draw (A) -- +(0,-0.4);
\draw (A1) -- +(0,-0.4);
\draw (A2) -- +(0,-0.4);
\draw (A3) -- +(0,-0.4);
\draw (C) -- +(0,0.4);
\draw (C1) -- +(0,0.4);
\draw (C2) -- +(0,0.4);
\draw (C3) -- +(0,0.4);
\draw (A) -- +(-0.4,0);
\draw (B) -- +(-0.4,0);
\draw (C) -- +(-0.4,0);
\draw (A3) -- +(0.4,0);
\draw (B3) -- +(0.4,0);
\draw (C3) -- +(0.4,0);
\end{tikzpicture}
\caption{A perfect orientation which is a Kasteleyn orientation.}\label{fig:sg}
\end{figure}
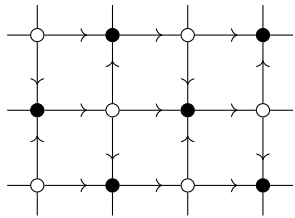
Consider a perfect orientation of a square grid on a torus shown in Figure \ref{fig:sg}. Here all black-to-white edges are parallel to each other, so the turning number of every white-to-black edge is equal to $1$. Therefore, the associated fractional Kasteleyn orientation is actually integral and coincides with the perfect orientation. The given perfect orientation is indeed Kasteleyn, since every face has exactly one black-to-white edge.
\end{example}

\begin{example}\label{ex:sg2}
 \begin{figure}[t]
 \centering
\begin{tikzpicture}[, scale = 1.2]
\node [draw,circle,color=black, fill=white,inner sep=0pt,minimum size=5pt] (A) at (0,0) {};
\node [draw,circle,color=black, fill=black,inner sep=0pt,minimum size=5pt] (B) at (0,1) {};
\node [draw,circle,color=black, fill=white,inner sep=0pt,minimum size=5pt] (C) at (0,2) {};
\node [draw,circle,color=black, fill=black,inner sep=0pt,minimum size=5pt] (A1) at (1,0) {};
\node [draw,circle,color=black, fill=white,inner sep=0pt,minimum size=5pt] (B1) at (1,1) {};
\node [draw,circle,color=black, fill=black,inner sep=0pt,minimum size=5pt] (C1) at (1,2) {};
\node [draw,circle,color=black, fill=white,inner sep=0pt,minimum size=5pt] (A2) at (2,0) {};
\node [draw,circle,color=black, fill=black,inner sep=0pt,minimum size=5pt] (B2) at (2,1) {};
\node [draw,circle,color=black, fill=white,inner sep=0pt,minimum size=5pt] (C2) at (2,2) {};
\node [draw,circle,color=black, fill=black,inner sep=0pt,minimum size=5pt] (A3) at (3,0) {};
\node [draw,circle,color=black, fill=white,inner sep=0pt,minimum size=5pt] (B3) at (3,1) {};
\node [draw,circle,color=black, fill=black,inner sep=0pt,minimum size=5pt] (C3) at (3,2) {};

\draw [->-] (A) -- (A1) node[midway, below] {$1$};  
\draw [->-] (A1) -- (A2) node[midway, below] {$-1$}; ;
\draw [->-] (A2) -- (A3) node[midway, below] {$1$}; ;
\draw [->-] (B1) -- (B) node[midway, below] {$1$};  
\draw [->-] (B2) -- (B1) node[midway, below] {$-1$};  ;
\draw [->-] (B3) -- (B2) node[midway, below] {$1$};  ;
\draw [->-] (C) -- (C1) node[midway, below] {$1$};  ;
\draw [->-] (C1) -- (C2) node[midway, below] {$-1$};  ;
\draw [->-] (C2) -- (C3) node[midway, below] {$1$};  ;
\draw [->-] (A) -- (B) node[midway, left] {$\mathrm i$};  ;
\draw [->-] (C) -- (B)  node[midway, left] {$-\mathrm i$};  ;;
\draw [->-] (A2) -- (B2) node[midway, left] {$\mathrm i$}; ;
\draw [->-] (C2) -- (B2)  node[midway, left] {$-\mathrm i$}; ;
\draw [->-] (B1) -- (A1) node[midway, left] {$\mathrm i$}; ;
\draw [->-] (B1) -- (C1) node[midway, left] {$-\mathrm i$}; 
\draw [->-] (B3) -- (A3) node[midway, left] {$\mathrm i$}; ;
\draw [->-] (B3) -- (C3)  node[midway, left] {$-\mathrm i$}; 
\draw (A) -- +(0,-0.4);
\draw (A1) -- +(0,-0.4);
\draw (A2) -- +(0,-0.4);
\draw (A3) -- +(0,-0.4);
\draw (C) -- +(0,0.4);
\draw (C1) -- +(0,0.4);
\draw (C2) -- +(0,0.4);
\draw (C3) -- +(0,0.4);
\draw (A) -- +(-0.4,0);
\draw (B) -- +(-0.4,0);
\draw (C) -- +(-0.4,0);
\draw (A3) -- +(0.4,0);
\draw (B3) -- +(0.4,0);
\draw (C3) -- +(0.4,0);
\end{tikzpicture}
\caption{A perfect orientation of the square grid and the associated Kasteleyn marking.}\label{fig:sg2}
\end{figure}
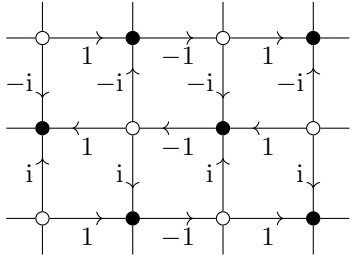
Figure \ref{fig:sg2} shows another perfect orientation of the square grid. Labels next to edges show the associated fractional Kasteleyn marking. The product of markings around every face is $-1$. It is also easy to see that the product of markings along any cycle is $\pm 1$. So this is indeed a Kasteleyn marking.
\end{example}

To prove Proposition \ref{prop:Kast}, we first establish a lemma. Say that a vertex $v$ of a face $f$ is a \textit{switch} if the edges of $f$ adjacent to $v$ have opposite orientations. 
\begin{lemma}\label{lemmaKas}
Let $bw(f)$ be the number of black-to-white edges of $f$. Then the number of switches is given by
$$
s(f) = l(f) - 2bw(f).
$$
\end{lemma}
\begin{proof}[\rm \bf Proof of the lemma]
First assume that there are no switches. Then the boundary of $f$ is a directed cycle half of whose edges are oriented white-to-black, and half black-to-white. So $l(f) = 2bw(f)$, and the desired formula holds. Now assume that there is at least one switch. Then the switches split the boundary of $f$ into $s(f)$ oriented intervals. Every interval starts at a white vertex and ends at a black one (it is important that we have no univalent vertices, otherwise this would not be true), and hence has an odd number $2k+1$ edges, $k$ of which are black-to-white. So, the contribution of each interval to the quantity $ l(f) - 2bw(f)$ is $1$, and  $ l(f) - 2bw(f)$ is equal to the number of intervals, i.e. $s(f)$.
\end{proof}

\begin{proof}[\rm \bf Proof of Proposition \ref{prop:Kast}]
Consider a face $f$. The alternating product of markings around $f$ is
$
(-1)^{bw(f)}
$
times a product of terms of the form $\mathrm{turn}(e_j)^{\pm 1}$ where $e_j$ are white-to-black edges of $f$. Furthermore, since every angle that contributes to the latter product is between edges that share a vertex, one can break that product down into terms corresponding to individual vertices. So, the alternating product of markings is $(-1)^{bw(f)}$ times a certain product over vertices. The contribution of each vertex $v$ depends on whether $v$ is a switch. If not, then the contribution of $v$ is $\exp(\frac{\mathrm i}{2}\mathrm{ext}(v))$ where $\mathrm{ext}(v)$ is the exterior angle of $f$ at $v$ (cf. Example \ref{exkm}). If $v$ is a switch, then the contribution is $\exp(-\frac{\mathrm i}{2}\mathrm{int}(v))$ where $\mathrm{int}(v) = \pi - \mathrm{ext}(v)$ is the interior angle at $v$. Since the sum of exterior angles is $2\pi$, the total contribution of all vertices is $-\exp\left(-\frac{\pi \mathrm i}{2} s(f)\right)$. By Lemma \ref{lemmaKas}, this rewrites as
$$
-\exp\left(-\frac{\pi \mathrm i}{2} (l(f) - 2bw(f))\right) = (-1)^{l(f) / 2 + bw(f) + 1}
$$
so the product of all markings around $f$ is indeed $(-1)^{l(f) / 2 + 1}$.\par 
For an arbitrary cycle, the argument is similar, but in that case the contribution of a switch is either $\exp(-\frac{\mathrm i}{2}\mathrm{int}(v))$, or $\exp(-\frac{\mathrm i}{2}(\mathrm{int}(v) - 2\pi))$, while the sum of exterior angles may be an arbitrary multiple of $2\pi$. So, in general the product of markings is $\pm 1$.
\end{proof}

\medskip

\section{The boundary path matrix and its characteristic polynomial}
\paragraph{Characteristic polynomial of the boundary path matrix for an arbitrary network.} Consider an arbitrary finite network (i.e. a weighted directed graph, not necessarily embedded) $\mathcal N$. The \textit{weighted path matrix} of $\mathcal N$ is the matrix whose $(i,j)$ entry is the formal sum of weights of all directed paths going from $i$'th to $j$'th vertex \cite{talaska2012determinants}. The entries of that matrix are formal series in terms of the weights which are actually rational functions. Indeed, let $A$ be the \textit{weighted adjacency matrix} of $\mathcal N$, i.e. the matrix whose $(i,j)$ entry is  the sum of weights of edges going from $i$'th to $j$'th vertex. Then the weighted path matrix is $ \Id + A + A^2 + \dots = (\Id - A)^{-1}$.\par

Now assume that $\mathcal N$ has $n$ distinguished labeled univalent sources and and $n$ distinguished labeled univalent sinks (there may be other sources and sinks as well, but in what follows by sources and sinks we mean these distinguished ones). We assume that the sources cannot be directly connected to  sinks. Define the \textit{boundary path matrix} $B$ as the matrix whose $(i,j)$ entry is the formal sum of weights of all directed paths going from $i$'th source to $j$'th  sink. It is a submatrix of the weighted path matrix. The proposition below gives a formula for the characteristic polynomial of $B$ in terms of certain adjacency matrices. Consider the network $\bar {\mathcal N}(\mu)$ obtained from $\mathcal N$ by gluing  sources to the corresponding  sinks, deleting the obtained $2$-valent vertices, and defining the weight of every newly formed edge as the product of weights of two edges of $\mathcal N$ it came from times $\mu$. Let $\bar A(\mu)$ be the weighted adjacency matrix of $\bar {\mathcal N}(\mu)$. 
\begin{proposition}\label{prop:initfla}
One has
$$
 \det( \Id - \mu B)=\frac{\det (\Id - \bar A(\mu))}{\det(\Id - \bar A(0))}.
$$
\end{proposition}
\begin{proof}
Call vertices of $\mathcal N$ that are neither sources nor sinks \textit{internal}. Let $m$ be their number. Let also $X$ be the $n \times m$ matrix whose entries are the weights of edges going from sources to internal vertices, and let $Y$ be the $m \times n$ matrix whose entries are the weights of edges going from internal vertices to sinks. Consider the network $\hat{\mathcal N}$ obtained from $\mathcal N$ by removing  sources and sinks (along with adjacent edges), and let $\hat A$ be its weighted adjacency matrix.
Then, by construction of the network $\bar {\mathcal N}(\mu)$, we have
\begin{equation}\label{eqbarA}
 \bar A(\mu) = \hat A + \mu YX.
\end{equation}
Let also $\hat W = (\Id - \hat A)^{-1}$ be the weighted path matrix of the network $\hat{\mathcal N}$. Then \eqref{eqbarA} implies
\begin{equation}\label{eqYX}
\hat W(\Id - \bar A(\mu)) = \hat W(\Id - \hat A)  -  \mu \hat WYX = \Id  -  \mu \hat WYX.
\end{equation}
Further, observe that any path in $\mathcal N$ going from a source to a sink can be uniquely represented as a concatenation of three paths: a path of length $1$ starting at a source, a path in $\hat{\mathcal N}$, and a path of length $1$ ending at a sink. Therefore, we have $B = X\hat WY$, and
\begin{equation}\label{eqXY}
\Id- \mu B = \Id - \mu X\hat WY.
\end{equation}
Now, using that $\det(\Id +  PQ) = \det(\Id +  QP)$ for any $n\times m$ matrix $P$ and $m \times n$ matrix $Q$, from \eqref{eqYX} and \eqref{eqXY} we get that
\begin{equation*}
\det(\hat W(\Id - \bar A(\mu))) = \det(\Id - \mu B).
\end{equation*}
Since $\hat W = (\Id - \hat A)^{-1}$ and $\hat A= \bar A(0)$, the result follows.
  \end{proof}
\paragraph{Characteristic polynomial of the boundary path matrix for a cylindric/toric network.} Now assume that $\mathcal N$ is embedded in a cylinder, with sources and sinks at the opposite boundary components. Assume also that $\mathcal N$ is endowed with a cut, i.e. a distinguished path connecting the boundary components which becomes a cycle when the boundary components are glued together in such a way that every source is identified with the corresponding sink. In that case, the weight of a path $\gamma$ is defined as the product of weights of its edges multiplied by $\lambda^{\mathrm{ind}(\gamma)}$ where the index ${\mathrm{ind}(\gamma)}$ of $\gamma$ is defined as the intersection number of $\gamma$ with the cut. The corresponding boundary path matrix $B(\lambda)$ can be viewed as the unsigned version of the boundary measurement matrix. Note that this definition of the boundary path matrix reduces to the one given above if we multiply the weights of edges of  $\mathcal N$ crossing the cut by $\lambda^{\pm 1}$, depending on the sign of the intersection. As a result, we obtain the following formula: 
\begin{corollary}\label{cor:bpm}
Consider a network $\mathcal N$ on a torus, with a chosen ideal rim and cut. Let $B(\lambda)$ be the boundary path matrix of the corresponding network on a cylinder, and let $  A(\lambda, \mu)$ be the adjacency matrix of the network ${\mathcal N}(\lambda, \mu)$ obtained from $\mathcal N$ by multiplying the weights of edges crossing the cut by $\lambda^{\pm 1}$ (depending on the sign of the intersection) and weights of edges crossing the rim by $\mu$.
Then
$$
 \det( \Id - \mu B(\lambda))=\frac{\det (\Id -  A(\lambda, \mu))}{\det(\Id -  A(\lambda, 0))}.
$$
\end{corollary}
\begin{proof}
Apply Proposition \ref{prop:initfla} to the network obtained from $\mathcal N$ by cutting the torus along the rim and multiplying the weights of edges crossing the cut by $\lambda^{\pm 1}$ (as usual, when an edge is cut into two, the weights of the newly formed edges are defined in such a way that their product is equal to the weight of the initial edge).
\end{proof}
\paragraph{Characteristic polynomial of the boundary path matrix for a bipartite perfect network.} 
In what follows, we will need a version of Corollary \ref{cor:bpm} for a bipartite perfect network. First, consider an arbitrary bipartite perfect network $\mathcal N$ (not necessarily embedded). Label the vertices in such a way that the unique edge starting at $i$'th black vertex ends at $i$'th white vertex. Define the \textit{bipartite adjacency matrix} $\mathcal A$ as follows: its $(i,j)$ entry is the sum of weights of edges connecting the $i$'th white vertex with $j$'th black vertex (the edges do not have to be oriented from white to black). Let $\mathcal A_{bw}$ be the diagonal part of $\mathcal A$ (corresponding to black-to-white edges), and $\mathcal A_{wb} = \mathcal A - \mathcal A_{wb}$ be the off-diagonal part (corresponding to white-to-black edges). Then, the weighted adjacency matrix of $\mathcal N$ (as defined above) is
$$
A = \left(\begin{array}{cc}0 & \mathcal A_{bw} \\\mathcal A_{wb} & 0\end{array}\right).
$$
So, 
\begin{equation}\label{ambn}
\det(\Id - A) = \det(\Id - \mathcal A_{bw}\mathcal A_{wb}),
\end{equation}
and the result of Corollary \ref{cor:bpm} can be restated as follows:
\begin{corollary}\label{cor:fin}
Consider a bipartite perfect network $\mathcal N$ on a torus, with a chosen ideal rim and cut. Let $B(\lambda)$ be the boundary path matrix of the corresponding network on a cylinder, and let $ \mathcal A(\lambda, \mu)$ be the bipartite adjacency matrix of the network ${\mathcal N}(\lambda, \mu)$ obtained from $\mathcal N$ by multiplying the weights of edges crossing the cut by $\lambda^{\pm 1}$ (depending on the sign of the intersection) and weights of edges crossing the rim by $\mu$.
Then 
$$
 \det( \Id - \mu B(\lambda))=\frac{\det (\Id -  {\mathcal A_{bw}}(\lambda, \mu) {\mathcal A}_{wb}(\lambda, \mu))}{\det(\Id -  {\mathcal A_{bw}}(\lambda, 0) {\mathcal A}_{wb}(\lambda, 0))},
$$
where $ \mathcal A_{bw}(\lambda, \mu)$ is the diagonal part of  $ \mathcal A(\lambda, \mu)$ (corresponding to black-to-white edges), while  $ \mathcal A_{wb}(\lambda, \mu)$ is the off-diagonal part of  $ \mathcal A(\lambda, \mu)$ (corresponding to white-to-black edges).\end{corollary}
\begin{proof}
This follows from \eqref{ambn} and  Corollary \ref{cor:bpm}.
\end{proof}
\medskip
\section{Proof of the main result}\label{sec:pmr}
In this section we prove Theorem \ref{thm1}. The strategy of the proof is as follows. We first show that the boundary measurement matrix of a perfect toric network is equal to the boundary path matrix of the same network but with modified weights. This allows us to use the formula provided by Corollary \ref{cor:fin} to express the characteristic polynomial of the boundary measurement matrix. The second step is to relate the right-hand side of that formula to the characteristic polynomial of the dimer model. That is done using Proposition \ref{prop:Kast}. \par
As explained in the introduction, it is sufficient to prove Theorem \ref{thm1} for a bipartite network drawn on a flat torus. Moreover, one can assume that all edges crossing the rim are parallel to each other and oriented from black to white.

\paragraph{The boundary measurement matrix as a path matrix.} The difference between the boundary measurement matrix $M(\lambda)$ and the boundary path matrix $B(\lambda)$ is the presence of signs in the definition of the former. Specifically, the contribution of a path $\gamma$ to the boundary path matrix is $\lambda^{\mathrm{ind}(\gamma)}\mathrm{wt}(\gamma)$, while its contribution to the boundary measurement matrix is  $(-1)^{c(\gamma)}\lambda^{\mathrm{ind}(\gamma)}\mathrm{wt}(\gamma)$. The sign $(-1)^{c(\gamma)}$ is defined in \cite{GSTV} using the following inductive construction. Let $\gamma$ be a path going from a source to a sink. Then:
\begin{enumerate}
\item If $\gamma$ is simple (i.e. does not cross itself),  its sign is $(-1)^{\mathrm{ind}(\gamma)}$.
\item If $\gamma$ can be decomposed into a path $\gamma'$ and a simple cycle, the signs of $\gamma$
and $\gamma'$ are opposite.
\end{enumerate}

Here we use a modified version of this definition. Namely, we assume that:
\begin{enumerate}
\item If $\gamma$ is simple on the universal covering of the cylinder, then its sign is $+1$.
\item If $\gamma$ can be decomposed into a path $\gamma'$ and a contractible simple cycle, then the signs of $\gamma$
and $\gamma'$ are opposite.
\end{enumerate}
The boundary measurement matrices constructed using these two definitions differ by a substitution $\lambda \mapsto -\lambda$. Such a transformation amounts to changing the spin structure and does not affect the result of Theorem \ref{thm1}.
\begin{proposition}\label{bmpm}
The sign of a path going from a source to a sink is equal to the product of turning numbers of all white-to-black edges on that path. 
\end{proposition}
 \begin{figure}[b]
 \centering
\begin{tikzpicture}[, scale = 0.9, decoration = {snake,   
                    pre length=3pt,post length=7pt,
                    },]

\node [draw,circle,color=black, fill=white,inner sep=0pt,minimum size=5pt] (A) at (1,3) {};
\draw [-<-]   (A) -- +(-1,0) coordinate (SRC);
\node [draw,circle,color=black, fill=gray,inner sep=0pt,minimum size=2pt, label = source] (SRCN) at (SRC) {};
\path  (A) -- +(2,2) coordinate (Bp);
\node [draw,circle,color=black, fill=black,inner sep=0pt,minimum size=5pt] (B) at (Bp) {};
\path []  (B) -- +(1,0) coordinate (SINK);
\node [draw,circle,color=black, fill=gray,inner sep=0pt,minimum size=2pt, label = sink] (SINKN) at (SINK) {};
\draw [->-]  (B) -- (SINKN);
\draw [->, decorate] (A) -- (B);
\path [->]  (SINKN) -- +(0,-3) coordinate (C);

\node [draw,circle,color=black, fill=white,inner sep=0pt,minimum size=5pt] (A) at (6,0) {};
\draw [-<-]   (A) -- +(-1,0) coordinate (SRC);
\node [draw,circle,color=black, fill=black,inner sep=0pt,minimum size=5pt] (SRCN) at (SRC) {};
\path  (A) -- +(2,2) coordinate (Bp);
\node [draw,circle,color=black, fill=black,inner sep=0pt,minimum size=5pt] (B) at (Bp) {};
\path []  (B) -- +(1,0) coordinate (SINK);
\node [draw,circle,color=black, fill=white,inner sep=0pt,minimum size=5pt] (SINKN) at (SINK) {};
\draw [->-]  (B) -- (SINKN);
\draw [->, decorate] (A) -- (B);
\path []  (SINKN) -- +(0,-3) coordinate (C);
\node [draw,circle,color=black, fill=black,inner sep=0pt,minimum size=5pt] (CN) at (C) {};
\draw [dashed, ->-]  (SINKN) -- (CN);
\path [->]  (SRCN) -- +(0,-1) coordinate (D);
\node [draw,circle,color=black, fill=white,inner sep=0pt,minimum size=5pt] (DN) at (D) {};
\draw [dashed, ->-]  (DN) -- (SRCN);
\draw [dashed, ->-] (CN) -- (DN);

\draw [->] (4,3) -- (5,2);

\end{tikzpicture}
\caption{Closing up a path.}\label{Kas2}
\end{figure}
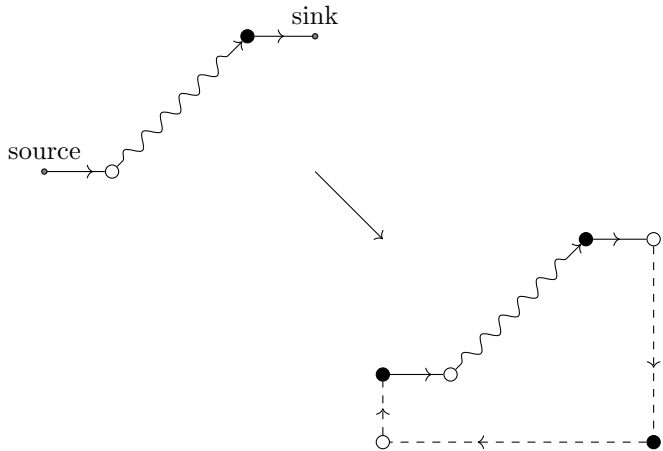
\begin{proof}
It is sufficient to check that the product of turning numbers of all white-to-black edges is equal to $+1$ for source-to-sink paths that are simple on the universal covering, and $-1$ for contractible simple cycles. Given a simple path, one can close it up to a cycle as shown in Figure~\ref{Kas2} (recall that vertices connected to sources are always white, vertices connected to sinks are black, and all edges connected to sources and sinks are parallel to each other). The turning numbers of both newly created white-to-black edges are equal to $\mathrm i$, so the product of turning numbers for the so created cycle is equal to the negative product of turning numbers for the initial path. Therefore, it suffices to show that for a simple contractible cycle the product of turning numbers of white-to-black edges is $-1$. The latter is proved using the same argument as in Example \ref{excycle}:  for a white-to-black edge $e$ in a directed cycle, we have $$\mathrm{turn}(e) = \exp\left(\frac{\mathrm i}{2}(\mathrm{ext}_+(e) + \mathrm{ext}_-(e))\right),$$ where $\mathrm{ext}_\pm(e)$ are exterior angles adjacent to $e$. And since the sum of all exterior angles is $2\pi$, the product of turning numbers is $\exp(\pi \mathrm i) = -1$, as desired. \qedhere

\end{proof}

\begin{corollary}\label{cor:bmm}
We have
$$
 \det( \Id - \mu M(\lambda))=\frac{P(\lambda, \mu)}{P(\lambda, 0)}$$
where
$$P(\lambda, \mu) = \det (\Id -  {\mathcal A_{bw}}(\lambda, \mu) {\mathcal A}^\mathrm{turn}_{wb}(\lambda, \mu)),$$ and  the matrix $\mathcal A^\mathrm{turn}_{wb}(\lambda, \mu)$ is obtained from the matrix  ${\mathcal A}_{wb}(\lambda, \mu)$ by means of multiplying the weight of every edge by its turning number.
\end{corollary}

\begin{proof}
Indeed, in view of Proposition \ref{bmpm} the boundary measurement matrix of the initial network is equal to the boundary path matrix for the network obtained from the initial one by multiplying the weights of all white-to-black edges by their turning numbers, so the desired formula follows from Corollary \ref{cor:fin}.
\end{proof}

\paragraph{Proof of Theorem \ref{thm1}.}
Since the network $\mathcal N$ is bipartite, its associated bipartite graph $\Gamma$ is just $\mathcal N$ itself. So, the isomorphism $\Psi \colon \Hom^1(\mathcal N, \C^*)\to\Hom^1(\Gamma, \C^*)$ is induced by the identity map on $1$-chains. Furthermore, since the black-to-white edges of $\mathcal N$ are oriented from white to black when viewed as edges of the bipartite graph $\Gamma$, the identity map on $1$-chains amounts to the following map between the edge weight spaces:  replace the weights of all black-to-white edges by their reciprocals, while keeping the weights of white-to-black edges intact. We need to show that this map takes the characteristic polynomial $ \det( \Id - \mu M(\lambda))$ of the boundary measurement matrix  to the rational function $K(\lambda, \mu) / K(\lambda, 0)$. To that end it suffices to show that the function 
$P(\lambda, \mu)$ from Corollary~\ref{cor:bmm}, is mapped to the characteristic polynomial $K(\lambda, \mu)$ of $\Gamma$, up to a monomial factor. Pushing forward $P(\lambda, \mu)$ by the map between edge weight spaces, we get the function
$$
\tilde P(\lambda, \mu) = \det (\Id -  {\mathcal A^{-1}_{bw}}(\lambda^{-1}, \mu^{-1}) {\mathcal A}^{\mathrm{turn}}_{wb}(\lambda, \mu)), 
$$
where we used that $\mathcal A_{bw}$ is a diagonal matrix, so inverting the weights is the same as inverting the matrix along with $\lambda$ and $\mu$. Up to a monomial factor, $\tilde P(\lambda, \mu)$ is equal to the determinant of the matrix
$$
S(\lambda, \mu):= {\mathcal A}^{\mathrm{turn}}_{wb}(\lambda, \mu) -  {\mathcal A_{bw}}(\lambda^{-1}, \mu^{-1}).
$$ So, to complete the proof it suffices to show that $S(\lambda, \mu)$ is precisely  the magnetically altered Kasteleyn matrix~\eqref{makm} for the Kasteleyn marking given by Proposition \ref{prop:Kast}. Consider first the off-diagonal part of $S$, i.e. $ {\mathcal A}^{\mathrm{turn}}_{wb}(\lambda, \mu) $. Note that this matrix actually does not depend on $\mu$, because all edges intersecting the rim are black-to-white. So, the $(i,j)$ entry of $S$ for $i \neq j$ is
$$
 S_{ij}(\lambda, \mu) = \sum\nolimits_e \mathrm{turn}(e) \mathrm{wt}(e) \lambda^{\langle e , \gamma_2 \rangle} 
$$
where the sum is taken over all edges going from $i$'th white to $j$'th black vertex, and $\gamma_2$ is the cut. Note that since every edge entering this sum is white-to-black, we have $k(e) = \mathrm{turn}(e) $, and thus $S_{ij}(\lambda, \mu) = \mathcal K_{ij}(\lambda, \mu)$. Consider now the diagonal part of $S$. By construction of the matrix ${\mathcal A_{bw}}(\lambda^{-1}, \mu^{-1})$, we have 
$$
S_{ii}(\lambda, \mu) = - \mathrm{wt}(e_i)  (\lambda^{-1})^{\langle e_i , \gamma_2 \rangle} (\mu^{-1})^{\varepsilon(e_i)}
$$
where $e_i$ is the unique edge from $i$'th black vertex to $i$'th white vertex, and $\varepsilon(e_i) = 1$ if $e_i$ intersects the rim and~$0$ otherwise. Let $\gamma_1$ be the rim. Then, since $\langle \gamma_1, \gamma_2 \rangle = 1$ and the cut goes in the direction from sources to sinks, it follows that  $\langle \gamma_1, e_i \rangle = 1$ for every edge $e_i$ intersecting the rim. So, 
$
\varepsilon(e_i) = {\langle \gamma_1, e_i \rangle}
$
for every black-to-white edge $e_i$. Finally, notice that since $e_i$ is oriented from black-to-white, the corresponding canonically oriented edge of $\Gamma$ is $-e_i$. Rewriting the formula for $S_{ii}$ as
$$
S_{ii}(\lambda, \mu) = - \mathrm{wt}(e_i)  \lambda^{\langle -e_i , \gamma_2 \rangle} \mu^{\langle \gamma_1, -e_i \rangle}.
$$
and taking into account that $k(e_i) = -1$, we see that $S_{ii}(\lambda, \mu) = \mathcal K_{ii}(\lambda, \mu)$. So we indeed have $S(\lambda, \mu) = \mathcal K(\lambda, \mu)$, which completes the proof of Theorem \ref{thm1}.\qed
\medskip

\bibliographystyle{plain}
\bibliography{networks.bib}

\begin{thebibliography}{10}

\bibitem{AGPR}
N.~Affolter, M.~Glick, P.~Pylyavskyy, and S.~Ramassamy.
\newblock Vector-relation configurations and plabic graphs.
\newblock {\em Sém. Lothar. Combin.}, 84{B}, 2020.

\bibitem{CR}
D.~Cimasoni and N.~Reshetikhin.
\newblock Dimers on surface graphs and spin structures. {I}.
\newblock {\em Comm. Math. Phys.}, 275(1):187--208, 2007.

\bibitem{FG}
V.~Fock and A.~Goncharov.
\newblock Cluster ensembles, quantization and the dilogarithm.
\newblock {\em Ann. Sci. \'Ec. Norm. Sup{\'e}r.}, 42(6):865--930, 2009.

\bibitem{FM}
V.V. Fock and A.~Marshakov.
\newblock Loop groups, clusters, dimers and integrable systems.
\newblock In {\em Geometry and quantization of moduli spaces}, pages 1--65.
  Springer, 2016.

\bibitem{FZ}
S.~Fomin and A.~Zelevinsky.
\newblock Cluster algebras {I}: {F}oundations.
\newblock {\em J. Amer. Math. Soc.}, 15(2):497--529, 2002.

\bibitem{GSTV}
M.~Gekhtman, M.~Shapiro, S.~Tabachnikov, and A.~Vainshtein.
\newblock Integrable cluster dynamics of directed networks and pentagram maps.
\newblock {\em Adv. Math.}, 300:390--450, 2016.

\bibitem{GSV}
M.~Gekhtman, M.~Shapiro, and A.~Vainshtein.
\newblock Cluster algebras and {P}oisson geometry.
\newblock {\em Mosc. Math. J.}, 3(3):899--934, 2003.

\bibitem{GSV2}
M.~Gekhtman, M.~Shapiro, and A.~Vainshtein.
\newblock Poisson geometry of directed networks in a disk.
\newblock {\em Selecta Math.}, 15(1):61--103, 2009.

\bibitem{GSV3}
M.~Gekhtman, M.~Shapiro, and A.~Vainshtein.
\newblock Poisson geometry of directed networks in an annulus.
\newblock {\em J. Eur. Math. Soc.}, 14(2):541--570, 2012.

\bibitem{glick2011pentagram}
M.~Glick.
\newblock The pentagram map and {Y}-patterns.
\newblock {\em Adv. Math.}, 227(2):1019--1045, 2011.

\bibitem{GK}
A.B. Goncharov and R.~Kenyon.
\newblock Dimers and cluster integrable systems.
\newblock {\em Ann. Sci. \'Ec. Norm. Sup{\'e}r.}, 46(5):747--813, 2013.

\bibitem{Kas}
P.~Kasteleyn.
\newblock The statistics of dimers on a lattice: {I}. {T}he number of dimer
  arrangements on a quadratic lattice.
\newblock {\em Physica}, 27(12):1209--1225, 1961.

\bibitem{KO}
R.~Kenyon and A.~Okounkov.
\newblock Planar dimers and {H}arnack curves.
\newblock {\em Duke Math. J.}, 131(3):499--524, 2006.

\bibitem{KOS}
R.~Kenyon, A.~Okounkov, and S.~Sheffield.
\newblock Dimers and amoebae.
\newblock {\em Ann. of Math.}, 163:1019--1056, 2006.

\bibitem{Pos}
A.~Postnikov.
\newblock Total positivity, {G}rassmannians, and networks.
\newblock {\em arXiv:math/0609764}, 2006.

\bibitem{Sch}
R.~Schwartz.
\newblock The pentagram map.
\newblock {\em Exp. Math.}, 1(1):71--81, 1992.

\bibitem{talaska2012determinants}
K.~Talaska.
\newblock Determinants of weighted path matrices.
\newblock {\em arXiv:1202.3128}, 2012.

\end{thebibliography}

\end{document}